\pdfoutput=1
\documentclass[review]{siamart171218}

\ifpdf
\DeclareGraphicsExtensions{.eps,.pdf,.png,.jpg}
\else
\DeclareGraphicsExtensions{.eps}
\fi

\newcommand{\TheTitle}{Computation of real-valued basis functions which transform as irreducible representations of the polyhedral groups} 
\newcommand{\TheTitles}{Real basis functions for the polyhedral groups} 
\newcommand{\TheAuthors}{Nan Xu and Peter C. Doerschuk}

\headers{\TheTitles}{\TheAuthors}

\title{{\TheTitle}}

\author{
	Nan Xu\thanks{Wallace H. Coulter Department of Biomedical Engineering, Georgia
		Institute of Technology and Emory University, Atlanta, Georgia 30332, USA;
	School of Electrical and Computer Engineering, Cornell University, Ithaca, NY, USA;
		 ID (\email{im.nan.xu@gmail.edu}).}
	\and
	Peter C. Doerschuk\thanks{School of Electrical and Computer Engineering, Nancy E. and Peter C. Meinig School of Biomedical Engineering, Cornell University, Ithaca, NY, USA, ID (\email{pd83@cornell.edu}).
	}
}

\usepackage{caption}
\usepackage{amsfonts}
\usepackage{amsmath,mathtools,amssymb}
\usepackage{algorithm,algorithmic} 
\usepackage[algo2e,linesnumbered,ruled,vlined]{algorithm2e} 

\usepackage{color,soul}
\usepackage{hang}
\usepackage{graphics,graphicx,epstopdf}
\usepackage{xcolor}
\usepackage[utf8]{inputenc}
\usepackage{tabularx}
    \newcolumntype{L}{>{\raggedright\arraybackslash}X}
\usepackage{multirow}
\usepackage{amsopn}

\newcommand{\reals}{{\mathrm{I\kern-.2em R}}}
\newcommand{\complex}{{\mathrm{C\kern-.6em C}}}
\newcommand{\field}{{\mathrm{I\kern-.2em F}}}
\newcommand{\expectation}{{\mathrm{I\kern-.2em E}}}

\newcommand{\dd }{{\rm d}}
\newcommand{\tr}{{\rm tr}}

\newcommand{\calC}{{\cal C}}
\newcommand{\calD}{{\cal D}}

\newcommand{\calH}{{\cal H}}

\newcommand{\calP}{{\cal P}}

\newcommand{\jp}{j^\prime}

\newcommand{\lp}{l^\prime}

\newcommand{\m}{m^\prime}

\newcommand{\np}{n^\prime}

\newcommand{\pp}{p^\prime}

\newcommand{\vx }{{\bf x}}

\newcommand{\basispln}{F_{p,l,n}}

\newcommand{\Nirrep}{N_{\rm rep}}
\newcommand{\Ngroup}{N_g}
\newcommand{\Npl}{N_{p;l}}

\newcommand{\indicator}{\chi}
\DeclareMathOperator{\diag}{diag}

\nolinenumbers
\begin{document}
\maketitle
\begin{abstract}
 Basis functions which are invariant under the operations of a rotational point group $G$ are able to describe any 3-D object which exhibits the rotational point group symmetry. However, in order to characterize the spatial statistics of an ensemble of objects in which each object is different but the statistics exhibit the symmetry, a complete set of basis functions is required. In particular, for each irreducible representation (irrep) of $G$, it is necessary to include basis functions that transform according to that irrep. This complete set of basis functions is a basis for square-integrable
 functions on the surface of the sphere in 3-D. Because the objects are real-valued, it is convenient to have real-valued basis functions. In this paper the existence of such real-valued bases is proven and an algorithm for their computation is provided for the icosahedral $I$ and the octahedral $O$ symmetries. Furthermore, it is proven that such a real-valued basis does not exist for the tetrahedral $T$ symmetry because some irreps of $T$ are essentially complex. The importance of these basis functions to computations in single-particle cryo electron microscopy is described.
\end{abstract}

\begin{keywords}
	polyhedral groups,
	real-valued matrix irreducible representations,
	numerical computation of similarity transformations,
	basis functions that transform as a row of an irreducible representation of
	a rotation group,
	Classical groups (02.20.Hj),
	Special functions (02.30.Gp).
\end{keywords}

\begin{AMS}
  2604, 2004, 57S17, 57S25
\end{AMS}

\section{Introduction}
\label{sec:intro}
A finite group of rotational symmetries of a 3-D object \textcolor{black}{(i.e., a geometry object in the 3D Cartesian space such as a platonic solid)}, denoted by $G$, arises in several situations including quasi-crystals~\cite{suck2002quasicrystals}, fullerenes~\cite{schwerdtfeger2015topology}, and viruses~\cite{horne1961symmetry}. By definition, a finite group is a set of finite elements equipped with a binary operation that combines any two elements to form a third element in such a way that four conditions, namely closure, associativity, identity and invertibility, are satisfied. Group elements $g\in G$ can be represented by matrices, namely representation matrices. Specifically, a group representation is an invertible linear transformation from the group elements to a set of representation matrices so that the group operation can be represented by matrix multiplication. An irreducible representation (irrep) of a group is a group representation that cannot be further decomposed into nontrivial invariant subspaces. The trivial irrep or identity irrep has all representation matrices to be identity (e.g., ``1''). An unitary irrep has all irrep matrices to be unitary. Two irreps are inequivalent if it is impossible to find a similarity transform relating them. More terminology definitions can be found in \cite{johnsonrepresentationterm}.
\par
One method for representing a 3-D object such as a quasi-crystal, a fullerene or a virus particle is an orthonormal expansion in basis functions where each basis function is a basis vector that associates the group operation with its matrix representation~\cite[Section 5.1]{Cornwell1984}. \textcolor{black}{Specifically, the electron scattering intensity of the 3-D object at coordinate $\vx\in\mathbb{R}^3$, denoted by $\rho(\vx)$, can be expressed by a Fourier series $\phi_{p,j}(\vx)$ with coefficients $w_{p,j}$ which are random variables, i.e.,
\begin{equation}\label{eq:ScatteringIntensity}
	\rho(\vx)=\sum_p\sum_j w_{p,j} \phi_{p,j}(\vx),
\end{equation}
where $p$ indexes the $p$th irrep.} If the object is invariant under the operations of $G$, then each basis function should transform according to the identity irrep of $G$ \textcolor{black}{(i.e., $p=1$ in Eq.~\ref{eq:ScatteringIntensity})} and such basis functions, called ``invariant basis'', have been extensively
studied~\cite{FoxOzierTetrahedralHarmonics1970,MuggliCubicHarmonics1972,AltmannPCambPhilSoc1957,MeyerCandJMath1954,CohanPCambPhilSoc1958,ElcoroPerezMatoMadariaga1994,HeuserHofmannZNaturforsch1985,JackHarrison1975,KaraKurkiSuonioActaCryst1981,LaporteZNaturforschg1948,ConteRaynalSoulie1984,RaynalConte1984,ZhengDoerschukActaCryst1995,FernandoJComputPhys1994}.
In more complicated scenarios, the object is not invariant under the operations of $G$. In particular, to describe any random 3-D object, a complete set of basis spanning the $L^2$ space is required, whereas such basis can be constructed by the basis functions that transform according to each of the irreps of $G$ (``all irreps basis'')~\cite[p.65, Theorem 1]{cornwell1997group}. Such basis functions have also been
studied~\cite{CohanPCambPhilSoc1958,MuggliCubicHarmonics1972,ConteRaynalSoulie1984,RaynalConte1984,BellonDodecahedralCosmologyClassicalQuantumGravity2006}. Our motivating problem, a structural biology problem described in \textcolor{black}{ section~\ref{sec:motivation}},
is an example of the more complicated situation. 
\par
\textcolor{black}{Note that the Fourier series $\phi_{p,j}(\vx)$ can be a product of an angular basis function $F_{p,j}(\vx/x)$ and a radial basis function $h_{p,j}(x)$, i.e., $\phi_{p,j}(\vx)=F_{p,j}(\vx/x)h_{p,j}(x)$~\cite[Appendex A.1]{xu2018allosteric}. One natural choice for the radial basis
functions $h_{p,j}(x)$ is exactly the family of Spherical Bessel functions~\cite{ZhengDoerschukIP1996}, which form a complete orthonormal set on $\mathbb{R}^+\cup\{0\}$. Then, computing a set of desired angular basis functions $F_{p,j}(\cdot)$ is the focus of this paper. In the remainder of this paper, the word ``basis functions'' (or ``basis'') all refer to the angular basis functions (or the angular basis).} In this paper, we provide a practical computational algorithm for a set of basis functions 
with the following four properties (which are specified in Eqs, \ref{eq:IfromY}-\ref{eq:PFprop}, respectively):
\begin{enumerate}
	\item
	Each function in the basis is a linear combination of spherical harmonics\footnote{Throughout this paper, spherical harmonics are denoted by $Y_{l,m}(\theta,\phi)$ where the degree $l$ satisfies $l\in\{0,1,\dots\}$, the order $m$ satisfies $m\in\{-l,\dots,l\}$ and $(\theta,\phi)$ are the angles of spherical coordinates with $0\leq \theta\leq \pi$ and $0\leq \phi\leq 2\pi$~\cite[Section~14.30, pp.~378--379]{OlverLozierBoisvertClark2010PATCH}.} of a fixed degree $l$.
	\item
	Each function in the basis is real-valued.
	\item
	The basis functions are orthonormal.
	\item
	\label{item:transformsasthenthrow}
	Under the rotations of a finite symmetry group, each function in the basis
	transforms as one row of the corresponding unitary irreducible representation (irrep)
	matrices.
\end{enumerate}
\par
Motivated by the study in structural virology, in which viruses often
exhibit the symmetry of a Platonic
solid~\cite{CrickWatsonStructureSmallVirusesNature1956}, we are especially
interested in the three polyhedral groups--including 1) the tetrahedral group $T$ that contains the 12 rotational symmetries of a regular tetrahedron, 
2) the octahedral group $O$ that contains the 24 rotational symmetries of a cube and a regular octahedron, and 3) the icosahedral group $I$ that contains the 60 rotational symmetries of a regular dodecahedron and a regular icosahedron.
For reasons that are described in Section~\ref{sec:computebasisfunctions}, in the cases of the octahedral and icosahedral groups, it is possible to find a set of basis functions which is complete in the space of
square-integrable functions on the surface of the sphere and which
satisfies Properties~1--4.
However, in the case of the tetrahedral group, it is not possible to find a set of basis functions that is both complete and which satisfies Properties~1--4. In such a situation, one way to achieve completeness is to add additional complex-valued functions which is an undesirable situation for our structural biology application (Section \ref{sec:motivation}).
\par
In the majority of existing literature, basis functions of a symmetry group have been generated as a linear combination of spherical harmonics of a single degree~\cite{AltmannPCambPhilSoc1957,AltmannBradleyHexagonal1965,AltmannCracknellCubicGroups1965,MuellerPriestleyCubicPhysRev1966,PuffCubicPhysicaStatusSolidiB1970,MuggliCubicHarmonics1972,FoxOzierTetrahedralHarmonics1970,ZhengDoerschukActaCryst1995,ZhengDoerschukTln1997}, because of the importance of rotations and the relative simplicity of
rotating spherical harmonics. Spherical harmonics have been widely applied in structural biology, e.g.,
the fast rotation function~\cite{CrowtherFastRotation1972}. Other work express the basis functions of a polyhedral group as multipole expansions in rectangular
coordinates~\cite{JahnMethaneSpectrumRoyalSocA1938,HechtTetrahedralJMolSpectroscopy1961}.
Previous work uses a variety of techniques and often has a restriction on
the value $l$ of the spherical
harmonics~\cite{AltmannPCambPhilSoc1957,AltmannBradleyHexagonal1965,AltmannCracknellCubicGroups1965,MuellerPriestleyCubicPhysRev1966,PuffCubicPhysicaStatusSolidiB1970,CohanPCambPhilSoc1958,MuggliCubicHarmonics1972,ConteRaynalSoulie1984,RaynalConte1984}. For instance,
Refs.~\cite{AltmannPCambPhilSoc1957,AltmannBradleyHexagonal1965,AltmannCracknellCubicGroups1965}
consider a range of point groups and use the techniques of projection
operators and Wigner $D$ transformations to compute basis functions up to
degree $l=12$, while Ref.~\cite{CohanPCambPhilSoc1958} uses similar
techniques restricted to the icosahedral group to provide basis functions
up to degree $l=15$.
Refs.~\cite{MuellerPriestleyCubicPhysRev1966,PuffCubicPhysicaStatusSolidiB1970}
use the method of representation transformation to compute the invariant
basis functions of the cubic group up to degree $l=30$; the work of
Ref.~\cite{MuggliCubicHarmonics1972} extends this computation to all irreps
basis functions. 
Refs. \cite{ConteRaynalSoulie1984,RaynalConte1984} propose a method for
deriving all irreps basis functions of the cubic and the icosahedral groups
for a specific degree $l$.
However, for computation which needs all irreps basis functions for a large
range of $l$ values (e.g., from 0 to 55), the one-by-one derivation is
cumbersome.
Later work~\cite{FoxOzierTetrahedralHarmonics1970,ZhengDoerschukActaCryst1995,ZhengDoerschukTln1997}
release this restriction on the degree $l$ and allow for the computation of
the invariant basis functions of any polyhedral group.
However, the recursions
in~\cite{ZhengDoerschukActaCryst1995,ZhengDoerschukTln1997} appear to be
unstable in computational experiments.
The cosmic topology implications of the Wilkinson Microwave Anisotropy
Probe observations have been
analyzed~\cite{CaillerieLanchiezeReyLuminetLihoucqRiazueloWeeksAstronAstrophys2007,BellonDodecahedralCosmologyClassicalQuantumGravity2006}
using functions similar to those of this paper for the icosahedral group.
The approach of
Ref.~\cite{BellonDodecahedralCosmologyClassicalQuantumGravity2006} is to
start with the invariant polynomials due to Felix Klein~\cite{Klein1884} 
and construct the desired functions by algebraic and differential operations on
polynomials.
In contrast, we start with known irreps, typically unitary, and use
linear algebra to compute real-valued unitary irreps. We then determine the desired functions from the real-valued generalized projection operators that are constructed from the real-valued irreps. Because of our application, we are focused on real-valued functions and we
show that such functions are only possible when there exists a real-valued irrep (Section~\ref{sec:realonly}, Lemma~\ref{prop:needrealirreps}).
Among the variety of symmetries that arise in our biophysical application,
essentially all rotational point group symmetries, there exist symmetries
that do not allow real-valued irreps and our approach makes this clear.
We also provide efficient software implementation of the proposed algorithm.
\par
In this paper, we derive an algorithm for efficiently computing the real-valued all irreps basis functions for the tetrahedral, octahedral, and icosahedral groups for arbitrary value of $l$. \textcolor{black}{The algorithm takes advantage of the exact solution calculated by \emph{Mathematica}~\cite{MathematicaURL} build-in functions (e.g., \texttt{WignerD}). Given these \emph{Mathematica}~\cite{MathematicaURL} build-in functions, our proposed method does not use any recurrence to compute the basis. This is to be contrasted with earlier work \cite{ZhengDoerschukActaCryst1995,ZhengDoerschukTln1997}, which computed basis functions by explicit recurrence relations that led to unstable results.} The most
burdensome calculation in the algorithm is to determine the eigenvectors of
a real symmetric matrix that is of dimension $2 d_p$ where $d_p$ is the
dimension of the $p$th irrep which, for the groups we consider, is no
larger than 5.
\par
To obtain the basis functions satisfying Properties~1--4, we first demonstrate that such basis functions exist if and only if real-valued irrep matrices exist (Section~\ref{sec:realonly}). 
Next, we determine the required real-valued irrep matrices
.
Standard approaches exist, e.g., Young
diagrams~\cite{FultonYoungTableaux1997}.
However, taking advantage of existing complex-valued irrep
matrices~\cite{AroyoKirovCapillasPerezMatoWondratschekBILBAOActaCryst2006,LiuPingChenJMathPhys1990},
we derive formula to find a similarity matrix that transforms the
complex-valued irrep matrices that are potentially-real (meaning that such a complex to real similarity matrix transformation exists) to real-valued irrep matrices (Section~\ref{sec:computingrealirrepmatrices}). Then, following the procedures as described in~\cite[p.~93]{Cornwell1984}, we determine real-valued generalized projection operators using the real-valued irrep matrices (obtained from Section~\ref{sec:computingrealirrepmatrices}), and apply them to real-valued spherical harmonics to obtain the desired basis functions (Section~\ref{sec:computebasisfunctions}). Finally, we provide numerical examples for the three polyhedral groups (Section~\ref{sec:application}).

\section{\textcolor{black}{Motivation and contribution}}\label{sec:motivation}
\par
The motivation for studying these functions is to characterize the 3-D
heterogeneity of a nanometer-scale biological particle (virus, ribosome, {\em etc.}) based on single-particle cryo electron microscopy (cryo EM)~\cite{SubramaniamKuhlbrandtHendersonRecentAdvancesInCryoEMIUCrJ2016,ChemistryNobelPrize2017}.
Single-particle cryo electron microscopy (cryo EM)~\cite{BaiMcMullanScheresTrendsBiochemicalSci2015,ChengGrigorieffPenczekWalzCell2015,ChengCell2015}
provides essentially a
noisy 2-D projection in an unknown direction of the 3-D electron scattering
intensity of a $10^1$--$10^2$~nm biological object. For studies with high spatial resolution, only one image is taken of each instance of the object because the electron beam rapidly damages the
object. There are multiple software systems, e.g.,
Refs.~\cite{FrankRadermacherPenczekZhuLiLadjadjLeithJStructBio1996,LudtkeBaldwinChiuJStructBio1999,ScheresRELIONimplementationJMB2012}, for computing a 3-D reconstruction of the object from sets of images of
different instances of the object and these systems include the possibility
that the instances come from a small set of classes where all instances
within one class are identical (\emph{homogeneous}). However, not only may there be multiple classes of heterogeneity, but each instance within
a class may vary due to, for example, flexibility \textcolor{black}{(\emph{continuous
heterogeneity})~\cite{taylor2008retrospective,BaiMcMullanScheres2015} (see also the report of the 2017 Nobel Prize in Chemistry \cite{nobelprizechemistry2017}).}
\par
Symmetry is an important characteristic of many biological particles. The Protein Data Bank~\cite{proteindatabank} contained
130,005 structures and 39\% had a rotational symmetry. \textcolor{black}{Some recent study~ \cite{YiliZhengQiuWangDoerschukJOSA2012} has relaxed the \emph{homogeneous} class assumption by merging the {symmetry} property of the biological object. Specifically, using the expression of the electron scattering intensity of the object in Eq.~\ref{eq:ScatteringIntensity}, the reconstruction algorithm imposes the
symmetry on $\rho(\vx)$, which achieves the assumption that all instances within each class are different but have identical symmetry ({\emph{symmetric individuals}})~\cite[Eqs.~55--56]{YiliZhengQiuWangDoerschukJOSA2012}. The invariant basis (setting $p=1$ in Eq.~\ref{eq:ScatteringIntensity}) becomes sufficient to achieve such an assumption. In the case of virus particles, for example, most of which exhibit icosahedral symmetry~\cite{CasperKlug1962}, the icosahedral basis functions associated to the identity irrep~\cite{ZhengDoerschukTln1997} have been popularly employed in the Fourier series~\cite{QiuWangMatsuiDomitrovicYiliZhengDoerschukJohnsonJSB2012,TangKearneyQiuWangDoerschukBakerJohnsonJMolRecog2014,YunyeGongVeeslerDoerschukJohnsonJSB2016}}. 
\par
\textcolor{black}{
Our goal is to further merge the ideas that biological particles obey {{symmetry}} and that different instances of the particle are heterogeneous due to different vibrational states (\emph{continuous heterogeneity}). In the continuous
heterogeneity situation, it becomes more natural to impose the symmetry on the {$1^{st}$}- and $2^{nd}$-order \emph{statistics} of $\rho(\vx)$ for the particle ({\emph{symmetric statistics}})~\cite{XuVeeslerDoerschukJohnsonHK97JSB2017,XuDoerschukIEEETransImageProc2019}, rather
than on the $\rho(\vx)$ itself. In this more realistic assumption, since only the statistics have symmetry, the individual particles may be non-symmetric, and therefore, the invariant basis becomes no longer sufficient. Instead, a complete basis are required.}
\par
\textcolor{black}{
The combined ideas of \emph{continuous heterogeneity} and {\emph{symmetric statistics}} require constraints on the mean and covariance of the coefficients vector $w$. As described in \cite[Section~V]{XuDoerschukIEEETransImageProc2019}, the constraints are simplest if each basis function transforms under rotations of the group as some row of some irrep of the group (Eq.~\ref{eq:PFprop}) and if all of the basis functions are real valued (Eq.~\ref{eq:I:isreal}). These two
goals are the primary topic of this paper. Using harmonic functions (Eq.~\ref{eq:IfromY}) helps characterize the spatial resolution of the estimated electron scattering intensity and leads to simple formulas for both the electron scattering intensity and the 3-D Fourier transform of the electron scattering intensity. Using orthonormal functions (Eq.~\ref{eq:I:orthonormal}) improves the numerical properties of the inverse problem.}
\par 
\textcolor{black}{
Our focus on real-valued basis functions comes from the fact that the electron scattering intensity $\rho(\vx)$ is real valued and the complete orthonormal radial functions $h(x)$ are also real valued (e.g., the Spherical Bessel functions). Therefore, if all the angular basis functions $F(\vx/x)$ are also real valued, then the coefficients $w$ can be real valued which simplifies the statistical estimation problem in two ways.  Suppose $w$ must be complex. The first complication is that it is necessary to estimate both the expectation of $w w^T$ and of $w w^H$.  The second complication is that it is necessary to account for constraints on the allowed values of $w$, much like a 1-D Fourier series for a real-valued function that is periodic with period $\Delta$ requires that the coefficients (denoted by $w_n$) satisfy $w_n=w_{-n}^\ast$ when the basis functions for the Fourier series are $\exp(i (2\pi/\Delta) n t)$.  Our focus on real-valued basis functions which allow real-valued coefficients permits us to avoid both of these complications for the important case of the icosahedral and octahedral groups}.
\par
\textcolor{black}{
With these unique properties, these basis functions that we study in this paper have been employed in the recent 3D image reconstruction calculations~\cite{xu2018allosteric,XuDoerschukIEEETransImageProc2019}, which eliminated the well-recognized long-standing distortions on and near symmetry axes of the biological object that were reconstructed by previous calculations  \cite[p.173]{Ludtke.MethodsEnzymology.2016} (see also~\cite{ZhengWangDoerschukJOSA2012,WangFuKhayatDoerschukJohnson2011,domitrovic2013virus}) in which only the functions of invariant basis were used. This has allowed important biological functions to be discovered along the symmetry axes of the virus particles~\cite{xu2018allosteric}. Furthermore, using all basis functions of this type dramatically reduces the number of parameters that must be estimated from the image
data~\cite[Figure~2]{XuDoerschukIEEETransImageProc2019} and makes each
parameter independent of the other parameters. An estimator might represent the electron scattering intensity $\rho(\vx)$ as a weighted sum of some alternative set of functions, e.g., as a 3-D array of voxels. Even in that case, the functions described in this paper would still be important, because it is likely that they would be involved in describing the constraint on the statistics of the weights for the alternative set of functions.}

\section{Notation}
\par
The following notation is used throughout the paper.
Let $M$ be a matrix with entry of $i$th row and $j$th column denoted by $(M)_{i,j}$. Then $M^\ast$ is the complex conjugate of $M$, $M^T$
is the transpose of $M$, and $M^H$ is the Hermitian transpose of $M$, i.e.,
${(M^T)}^\ast$. $I_n\in\mathbb{R}^{n\times n}$ is the identity matrix.
$\Re$ and $\Im$ are the real and imaginary parts, respectively, of their
arguments. For 3-D vectors, $x=\|\vx\|_2$ and $\vx/x$ is shorthand for the
$(\theta,\phi)$ angles in the spherical coordinate system.
Integration of a function $f:\mathbb{R}^3\rightarrow\mathbb{C}$ over the
surface of the sphere in $\mathbb{R}^3$ is denoted by $\int
f(\vx)\dd\Omega$ meaning $\int_{\theta=0}^{\pi} \int_{\phi=0}^{2\pi}
f(x,\theta,\phi)\sin\theta\dd\theta\dd\phi$.
The Kronecker delta function is denoted by $\delta_{i,j}$ and has value 1
if $i=j$ and value 0 otherwise.
\par ``Representation'' and ``Irreducible representation'' are abbreviated by
``rep'' and ``irrep'', respectively. \textcolor{black}{For the finite group $G$, let $\Gamma^p(g)\in\mathbb{C}^{d_p\times d_p}$ be the unitary irrep matrix of the $p$th irrep for the group element $g\in G$ with group order $N_g$, where $p\in\{1,\dots,\Nirrep\}$ indexes the inequivalent unitary irreps of $G$, and $\Nirrep$ is the total number of inequivalent irreps. Note that the values in the matrix $\Gamma^p(g)$ for all $g\in G$ may be either real or complex. In Section \ref{sec:computingrealirrepmatrices}, $\Gamma_c^p(g)\in\mathbb{C}^{d_p\times d_p}$ specifically denotes the complex-valued unitary irrep matrix, whereas $\Gamma_r^p(g)\in\mathbb{R}^{d_p\times d_p}$ denotes the real-valued orthonormal irrep matrix of the $p$th irrep.}  

\section{Real basis functions require and generate real irreps}
\label{sec:realonly}
The one result in this section, Lemma~\ref{prop:needrealirreps}, states
that a real-valued set of orthonormal basis functions of the $p$th irrep of
the finite group $G$ exists if and only if a real irrep exists, independent of
whether the basis functions are expressed as linear combinations of
spherical harmonics of fixed degree $l$. First of all, the basis functions of a group of coordinate transformations $G$ have the following definition:
\begin{definition}\label{def:basisfunction}
	(\cite[Eq.~1.26, p.~20]{Cornwell1984})
	A set of linearly independent functions $F_{p,1},\dots, F_{p,d_p}$, that associate with the $p^{th}$ irrep, form a basis of $G$, denoted by $ F_{p}(\cdot)=\left[\begin{smallmatrix}F_{p,1}(\cdot)\\
	\vdots\\
	F_{p,d_p}(\cdot) \end{smallmatrix}\right]:\mathbb{R}^3\rightarrow\mathbb{C}^{d_p}$, if for every $g\in G$,
	\begin{equation}\label{eq:I:rotationEntry0}
	P(g) F_{p,j}(\vx/x)
	=
	\sum_{m=1}^{d_p}(\Gamma^p(g))_{m,j}
	F_{p,m}(\vx/x),\text{~for $j=1,\dots, d_p$,}
	\end{equation}
	\begin{equation}\label{eq:I:rotationEntry}
	\mbox{
	or in the vector form,~~~}
	P(g) F_{p}(\vx/x)=
	(\Gamma^p(g))^T
	F_{p}(\vx/x),
	\end{equation}
	where $P(g)$ is the abstract rotation operator, i.e., $P(g)f(\vx)=	f(R_g^{-1}\vx)$, and $R_g\in\mathbb{R}^{3\times 3}$ with $R_g^{-1}=R_g^T$ and $\det R_g=+1$ is the rotation matrix corresponding to $g\in G$. When $P(g)$ is applied to a vector-valued function, it operates on each component of the vector.
	The function $F_{p,j}(\cdot)$ for $j\in\{1,\dots,d_p\}$ is said to ``transform as the $j^{th}$ row" of the $p^{th}$ irrep of the finite group $G$.
\end{definition}

\begin{lemma}\label{prop:needrealirreps}
	{Real-valued orthonormal basis functions of the $p^{th}$ irrep of the finite group $G$ exist if and only if the real-valued $p^{th}$ irrep exists for $G$.}
\end{lemma}
\begin{proof}
	Real-valued functions imply real-valued irreps:
	Let $F_{p,\zeta}(\vx/x)$ be the $\zeta$th orthonormal vector basis of $G$ that associate with the $p$th irrep (defined in Eq.~\ref{eq:I:rotationEntry}).
	
Define $J_{\zeta;\zeta^\prime}^{p;\pp}\in\mathbb{R}^{d_p\times d_p}$ by
	\begin{eqnarray}
	J_{\zeta;\zeta^\prime}^{p;\pp}
	&=&
	\int
	\left[P(g){
		F_{p,\zeta}
	}
	(\vx/x)\right]
	\left[
	P(g){
		 F_{\pp,\zeta^\prime}
	}
	(\vx/x)
	\right]^T
	\dd\Omega
	.
	\end{eqnarray}
	Evaluate $J_{\zeta;\zeta^\prime}^{p;\pp}$ twice.
	In the first evaluation,
	\begin{eqnarray}
	J_{\zeta;\zeta^\prime}^{p;\pp}
	&=&
	\int
	 F_{p,\zeta}
	(\vx/x)
	\left[
	 F_{\pp,\zeta^\prime}
	(\vx/x)
	\right]^T
	\dd\Omega=I_{d_p}\delta_{p,\pp}\delta_{\zeta,\zeta^\prime}
	,
	\end{eqnarray}
	where the first equality is due to rotation the coordinate system by $R_g$,
	and the second equality is due to the fact that the \{$F_{p,\zeta}$\} are
	orthonormal.
	\par
	In the second evaluation, use Eq.~\ref{eq:I:rotationEntry}, rearrange, and use
	the orthonormality of \{$F_{p,\zeta}$\} to get
	\begin{eqnarray}
	J_{\zeta;\zeta^\prime}^{p;\pp}
	&=&
	\int
	(\Gamma^p(g))^T
	 F_{p,\zeta}
	(\vx/x)
	\left[
	(\Gamma^{\pp}(g))^T
	 F_{\pp;\zeta^\prime}
	(\vx/x)
	\right]^T
	\dd\Omega
	\\
	&=&
	(\Gamma^p(g))^T
	\left[\int
	 F_{p,\zeta}
	(\vx/x)
	\left[
	 F_{\pp,\zeta^\prime}
	(\vx/x)
	\right]^T
	\dd\Omega
	\right]
	\Gamma^{\pp}(g)
	\\
	&=&
	(\Gamma^p(g))^T
	\left[I_{d_p}
	\delta_{p,\pp}
	\delta_{\zeta,\zeta^\prime}\right]
	\Gamma^{\pp}(g)
	\\
	&=&
	(\Gamma^p(g))^T
	\Gamma^{\pp}(g)
	\delta_{p,\pp}
	\delta_{\zeta,\zeta^\prime}
	.
	\end{eqnarray}
	Equating the two expressions for $J_{\zeta;\zeta^\prime}^{p;\pp}$ gives $
	(\Gamma^p(g))^T
	\Gamma^p(g)
	=
	I_{d_p}$.
	Since $\Gamma^p(g)$ is unitary, multiplying on the right by
	$({\Gamma^p(g)})^H$ implies that $({\Gamma^p(g)})^T=({\Gamma^p(g)})^H$ so
	that $\Gamma^p(g)$ is real.
	\par
	Real-valued irreps imply real-valued functions:
	This follows from the results in the later section, Lemma~\ref{lemma:projectionapplied} and
	Eqs.~\ref{eq:VectorBasisFunction2}--\ref{eq:MatrixBasisFunction}.
\end{proof}

\section{Computation of real irrep matrices}
\label{sec:computingrealirrepmatrices}
In  Section~\ref{sec:realonly}, we have proved that the real basis functions requires the real irreps. In this section, starting from a set of matrices that make up a complex-valued unitary irrep, we provide an approach to compute an equivalent real-valued orthonormal irrep if that's possible. The question of existence of such equivalent real-valued orthonormal irrep is answered by the Frobenious-Schur
theory~\cite[p.~129, Theorem~III]{Cornwell1984} (see also \cite[p.~708]{wightman1993collected}), which is summarized in the following paragraph. 
\par
The Frobenious-Schur indicator, denoted by $\indicator$, is defined as\vspace{-.6em}
\[\indicator(\{\Gamma_c^p(g)\}_{g\in G})=(1/\Ngroup)\sum_{g\in G} \tr[\Gamma_c^p(g)].\vspace{-.6em}\]
According to the Frobenious-Schur theory, the value of $\indicator$ is 1, 0, or -1 has the following implications:
\begin{enumerate}
    \item[a)] If $\indicator=1$, then the irrep is {\em potentially real}, meaning that there
exists a unitary matrix, denoted by $S^p$, such that $(S^p)^H \Gamma^p(g) S^p$ is real for all $g\in G$.
\item[b)] If $\indicator=0$, then the irrep is {\em essentially complex}, meaning that there
is no similarity transformation that relates $\Gamma^p$ and $(\Gamma^p)^\ast$.
\item[c)] If $\indicator=-1$, then the irrep is {\em pseudo real}, meaning that there exists a
unitary matrix, denoted by $T^p$, such that $(\Gamma^p(g))^\ast=(T^p)^H \Gamma^p(g) T^p$ for all $g\in G$, but no similarity transformation exists such that $(\Gamma^p(g))^\ast$ real for all $g\in G$.
\end{enumerate}
\par
The remainder of this section applies only to potentially real irreps (i.e. $\indicator=1$). Given the fact that the direct sum of the disjoint subspaces defined by the irreps of $G$ from the $L^2$ space \cite[p. 65-67]{cornwell1997group}, it is satisfactory for generating any set of orthonormal matrices that construct an irrep of $G$, and the question of uniqueness does not arise for our purpose of study. 
In the following, we describe a three-step algorithm to compute such a unitary matrix $S^p\in\mathbb{C}^{d_p\times d_p}$ for the case of potentially real irreps:
\begin{enumerate}
	\item For any such unitary matrix $S^p$, show that the complex irrep $\Gamma_c^p$ is similar to its complex conjugate $(\Gamma_c^p)^\star$ with the similarity transformation $S^p(S^p)^T$.
	\item Find a matrix $C^p$, which is an explicit function
	of $\Gamma_c^p$, and is a similarity matrix relating the two sets of matrices $\Gamma_c^p$ and $(\Gamma_c^p)^\star$.
	\item Factor $C^p$ to compute a particular $S^p$.
\end{enumerate}
\par
Step 1 
is achieved by Lemma~\ref{lemma:SStrans}.
\begin{lemma}
	\label{lemma:SStrans}
	Suppose that the $p$th irrep of the group $G$ which is represented by the complex unitary matrices $\Gamma_c^p(g)$ ($g\in G$) is potentially real.
	Let $S^p\in\mathbb{C}^{d_p\times d_p}$ denote a unitary matrix.
	The following two statements are equivalent:
	\begin{equation}\label{eq:complex2real}
	\mbox{For all $g\in G$,~}
	\Gamma_r^p(g)=(S^p)^H \Gamma_c^p(g) S^p~\mbox{such that $\Gamma_r^p(g)\in\mathbb{R}^{d_p\times d_p}$.}
	\end{equation}
	\begin{equation}
	\mbox{For all $g\in G$,~}
	[
	S^p
	(S^p)^T
	]^{-1}
	\Gamma_c^p(g)
	[
	S^p
	(S^p)^T
	]
	=
	(\Gamma_c^p(g))^\ast
	\label{eq:GammaSimilarToGammaConjugate}
	.
	\end{equation}
\end{lemma}
Please see Appendix~\ref{sec:appendixA} for the proof. 
\par
Step 2 computes a non-unitary symmetric matrix ($Z^p$) (Lemma~\ref{lemma:Z:def}), which is then normalized ($C^p$) to be unitary
(Corollary~\ref{corollary:Cp:def}).
\begin{lemma}
	\label{lemma:Z:def}
	Suppose that $\Gamma_c^p(g)$ ($g\in G$) are complex unitary irrep matrices for the $p$th rep of the group G which is potentially real.
	Let $A^p\in\mathbb{C}^{d_p\times d_p}$ be a nonsingular transpose-symmetric matrix (i.e., $(A^p)^T=A^p$) and $Z^p$ be defined by Eq.~\ref{eq:Z:def}, specifically,
	\begin{equation}
	Z^p
	=
	\frac{1}{\Ngroup}
	\sum_{g\in G}
	\Gamma_c^p(g)
	A^p
	({({\Gamma_c^p(g)})^\ast})^{-1}
	\label{eq:Z:def}
	.
	\end{equation}
	If $Z^p$ is nonzero, then $Z^p$ has the following properties:
	\begin{enumerate}
		\item
		$({Z^p})^T=Z^p$.
		\item
		${(Z^p)}^\ast
		Z^p
		=
		c_Z I_{d_p}$
		where $c_Z\in\mathbb{R}^+$.
		\item
		For all $g\in G$, $({\Gamma_c^p(g)})^\ast=({Z^p})^\ast\Gamma_c^p(g)Z^p$.
	\end{enumerate}
\end{lemma}
Please see Appendix~\ref{sec:appendixA} for the proof.
\par
It is important to find a matrix $A^p$ such that the matrix $Z^p$ is
nonzero.
For the three polyhedral groups that we consider in this paper,
this issue is discussed in Section~\ref{sec:platonicsolids}.
\begin{corollary}
	\label{corollary:Cp:def}
	Define $C^p$ by
	\begin{equation}
	C^p
	=
	Z^p/\sqrt{c_Z}
	\label{eq:Cp:def}
	.
	\end{equation}
	Then $C^p$ has the following properties:
	\begin{enumerate}
		\item
		\label{item:Cp:transpose}
		$({C^p})^T=C^p$.
		\item
		\label{item:Cp:hermitiantranspose}
		$
		{(C^p)}^\ast
		C^p
		=
		I_{d_p}
		$.
		\item
		\label{item:Cp:issimilaritymatrix}
		For all $g\in G$,
		$
		({\Gamma_c^p(g)})^\ast
		=
		({C^p})^\ast
		\Gamma_c^p(g)
		C^p
		$
		.
	\end{enumerate}
\end{corollary}
\par
The matrix $S^p$ in the definition of potentially real is not unique.
Comparing Property~\ref{item:Cp:issimilaritymatrix} of Corollary
\ref{corollary:Cp:def} and Eq.~\ref{eq:GammaSimilarToGammaConjugate}, $S^p$
can be restricted to satisfy
\begin{equation}
C^p=S^p{(S^p)}^T,
\label{eq:SisfactorofC}
\end{equation}
noting, however, that even with this restriction, $S^p$ is still not unique.
Because Lemma~\ref{lemma:SStrans} is ``if and only if'', any unitary matrix
$S^p$ that satisfies Eq.~\ref{eq:SisfactorofC} is a satisfactory similarity
matrix.
The existence of the unitary factorization described by
Eq.~\ref{eq:SisfactorofC} is guaranteed by the Takagi Factorization~\cite[Corollary~4.4.6, p.~207]{HornJohnson1985}.
\par
Step~3 is to perform the factorization of $C^p$ and a general algorithm is
provided by Lemma~\ref{lemma:eigenVSconeigen} which is based on the
relationship between the coneigenvectors (as is described in Property 3 of Lemma~\ref{lemma:eigenVSconeigen}) of a unitary symmetric matrix $Q$ and the eigenvectors of its real representation matrix $B$, which is
defined by
$B=\Bigl[\begin{array}{cc}
\Re{Q} & \Im{Q} \\ \Im{Q} & -\Re{Q}
\end{array} \Bigr]\in\mathbb{R}^{2n\times 2n}$.
\begin{lemma}
	\label{lemma:eigenVSconeigen}
	Let $Q\in\mathbb{C}^{n\times n}$ be a unitary symmetric matrix, i.e.,
	$Q^T=Q$ and $Q Q^\ast=I_n$.
	Let $B\in\mathbb{R}^{2n\times 2n}$ be the real representation of $Q$, i.e., 
	$B=\bigl[\begin{smallmatrix}
	\Re{Q} & \Im{Q} \\ \Im{Q} & -\Re{Q}
	\end{smallmatrix} \bigr]\in\mathbb{R}^{2n\times 2n}$.
	Then, the following properties hold:
	\begin{enumerate}
		\item
		$B$ is nonsingular and has $2n$ real eigenvalues and $2n$ orthonormal
		eigenvectors.
		\item
		The eigenvectors and eigenvalues of $B$ are in pairs, specifically,
		\begin{equation}
		B
		\left[
		\begin{smallmatrix}
		x \\
		-y
		\end{smallmatrix}
		\right]
		=\lambda
		\left[
		\begin{smallmatrix}
		x \\
		-y
		\end{smallmatrix}
		\right]
		\mbox{~if and only if~}
		B
		\left[
		\begin{smallmatrix}
		y \\
		x
		\end{smallmatrix}
		\right]
		=
		-\lambda
		\left[
		\begin{smallmatrix}
		y \\
		x
		\end{smallmatrix}
		\right]
		.\nonumber
		\end{equation}
		\item
		Let$\left[
		\begin{smallmatrix}
		x_1 \\
		-y_1
		\end{smallmatrix}
		\right]
		,
		\dots,
		\left[
		\begin{smallmatrix}
		x_n \\
		-y_n
		\end{smallmatrix}
		\right]$
		be the orthonormal eigenvectors of $B$ associated with $n$ positive
		eigenvalues of $\lambda_1, \dots, \lambda_n$.
		(Since $B$ is nonsingular, there are no zero eigenvalues.)
		Then $x_1-iy_1,\dots,x_n-iy_n$ are the set of orthonormal coneigenvectors
		of $Q$ associated with the $n$ coneigenvalues $+\lambda_k$, i.e.,
		$Q(x_k-iy_k)^\ast=\lambda_k(x_k-iy_k)$ for $k=1,...,n$.
		\item
		$\lambda_1=\dots=\lambda_n=1$.
		\item
		\label{item:U:def}
		Define $u_k=x_k-iy_k$ and $U=[u_1,\dots,u_n]\in\mathbb{C}^{n\times n}$.
		Then $U$ is unitary.
		\item
		$Q=U U^T$.
	\end{enumerate}
\end{lemma}
Please see Appendix~\ref{sec:appendixA} for the proof.
\par
Applying Lemma~\ref{lemma:eigenVSconeigen} to $C^p$ results in a particular
matrix $S^p$ which is the $U$ matrix of Property~5.
The complete algorithm is summarized in Theorem~\ref{theorem:realirrep}.
\begin{theorem}
	\label{theorem:realirrep}
	A unitary matrix, $S^p\in\mathbb{C}^{d_p\times d_p}$, which is a similarity
	transformation between the provided potentially-real complex unitary irrep
	and a real orthonormal irrep, can be computed by the following steps:
	\begin{enumerate}
		\item
		Compute $Z^p$ by Eq.~\ref{eq:Z:def}.
		\item
		Compute $c_Z$ by Lemma~\ref{lemma:Z:def} Property~3 and compute $C^p$ by
		Eq.~\ref{eq:Cp:def}.
		\item
		Compute the eigenvectors and eigenvalues of
		\begin{equation}
		B^p
		=
		\left[
		\begin{array}{cc}
		\Re C^p & \Im C^p \\
		\Im C^p & -\Re C^p
		\end{array}
		\right]
		\in
		\mathbb{R}^{2d_p\times 2d_p}
		.
		\end{equation}
		\item
		Form the matrix $V^p\in\mathbb{R}^{2d_p\times d_p}$ whose columns are the $d_p$
		eigenvectors of $B^p$ that have positive eigenvalues.
		\item
		Then $S^p=[I_{d_p}, iI_{d_p}]V^p$.
	\end{enumerate}
\end{theorem}

\section{Computation of real basis functions}
\label{sec:computebasisfunctions}
In this section, formulas corresponding to the four goals in
Section~\ref{sec:intro} are stated in Eqs.~\ref{eq:IfromY}--\ref{eq:PFprop}
and the computation of basis functions satisfying these formulas is then
described. Specifically, the basis functions which satisfy the four goals in Section~\ref{sec:intro}
have four indices:
$p$ indexes the unitary irreducible representation; 
$l$ indexes the subspace defined by spherical harmonics of fixed order $l$;
$n$ indexes the vector basis of $G$ that satisfies Eqn.~\ref{eq:I:rotationEntry}; and $j$ indexes the component of the vector basis.
Let $F_{p,l,n,j}$ be a basis function that transforms as the $j$th row of
the irrep matrices and $F_{p,l,n}=(F_{p,l,n,j=1},\dots,F_{p,l,n,j=d_p})^T$.
Let $Y_{l,m}(\theta,\phi)$ be the spherical harmonic of degree $l$ and
order $m$~\cite[Section~14.30,
pp.~378--379]{OlverLozierBoisvertClark2010PATCH}.
Then the goals are to obtain a set of functions such that
\begin{eqnarray}
F_{p,l,n,j}(\theta,\phi)
&=&
\sum_{m=-l}^{+l}
c_{p,l,n,j,m}
Y_{l,m}(\theta,\phi)
\label{eq:IfromY}
\\
F_{p,l,n,j}(\theta,\phi)
&\in&
\mathbb{R}
\label{eq:I:isreal}
\\
\delta_{p,\pp}
\delta_{l,\lp}
\delta_{n,\np}
\delta_{j,\jp}
&=&
\int_{\phi=0}^{2\pi}
\int_{\theta=0}^\pi
F_{p,l,n,j}(\theta,\phi)
F_{\pp,\lp,\np,\jp}(\theta,\phi)
\sin\theta\dd\theta\dd\phi
\label{eq:I:orthonormal}
\\
F_{p,l,n}(R_g^{-1}\vx/x)
&=&
(\Gamma_r^p(g))^T
F_{p,l,n}(\vx/x).
\label{eq:PFprop}
\end{eqnarray}
\par
The computation is performed by the projection method of
Ref.~\cite[p.~94]{Cornwell1984} in which various projection operators
(Definition~\ref{def:projectionoperator}) are applied to each function of a
complete basis for the space of interest.
When the various projection operators are defined using real-valued
orthonormal irrep matrices (as computed in
Section~\ref{sec:computingrealirrepmatrices}) and are applied to a
real-valued complete orthonormal basis in the subspace spanned by spherical
harmonics of degree $l$ (which has dimension $2l+1$) then the resulting
basis for the same subspace is real-valued, complete, and
orthonormal~\cite[Theorems~I and~II, pp.~92-93]{Cornwell1984}.
\par
The remainder of this section has the following organization.
First, the projection operators are defined (Definition~\ref{def:projectionoperator}). 
Second, the initial basis in the subspace is described. 
Third, the results of applying the projection operators to the basis functions are described in terms of individual functions (Lemma~\ref{lemma:projectionapplied}) and in terms of sparse matrices of order $(2l+1)\times (2l+1)$. 
Fourth, normalization is discussed (Eq.~\ref{eq:VectorBasisFunction2}). 
Fifth, because basis functions computed by this process are more than necessary (as is detailed in the later context), 
Gram-Schmidt orthogonalization is used to extract a orthonormal subset that spans the subspace defined by degree $l$.  
Finally, sixth, comments are made on the non-uniqueness of the final basis.
\par
\begin{definition}(\cite[p.~93]{Cornwell1984})
	\label{def:projectionoperator}
	The projection operators $\calP_{j,k}^p$ are defined by~
	\begin{equation}
	\calP_{j,k}^p
	=
	\frac{d_p}{\Ngroup}
	\sum_{g\in G}
	(\Gamma^p(g))_{j,k}^\ast
	P(g)
	\label{eq:projectionoperator:def}
	\end{equation}
	where 
	$P(g)$ is the abstract rotation operator as is defined in Definition \ref{def:basisfunction}.
\end{definition}
\par
The projection operator is applied to a complete set of basis functions.
One natural choice is the set of spherical harmonics~\cite[Eq.~14.30.1,
p.~378]{OlverLozierBoisvertClark2010PATCH} (denoted by
$Y_{l,m}(\theta,\phi)$, where the arguments will routinely be suppressed)
because $Y_{l,m}$ have simple rotational properties. Specifically, the $Y_{l,m}$ functions have the symmetry property
$Y_{l,-m}=(-1)^m Y_{l,m}^\ast$~\cite[Eq.~14.30.6,
p.~378]{OlverLozierBoisvertClark2010PATCH}
and the rotational property $P(R)
Y_{l,m}=\sum_{\m=-l}^{+l}D_{l,m,\m}(R)
Y_{l,\m}$,
where
$R$ is a rotation matrix, and
$D_{l,m,\m}(R)$ are the Wigner $D$ coefficients~\cite[Eq.~4.8,
p.~52]{Rose1957},
and $P(R)$ is the rotation operator $P(R)f(\vx)=f(R^{-1}\vx)$. However, except for $Y_{l,m=0}$, spherical harmonics are complex valued.
Older literature~\cite[Eq.~10.3.25, p.~1264]{MorseFeshbach1953}
used real-valued definitions, e.g.,
{\footnotesize
\begin{equation}
\check Y_{l,m}
=
\left\{
\begin{array}{ll}
\sqrt{2} \Im Y_{l,m} , & m<0 \\
Y_{l,0}, & m=0 \\
\sqrt{2} \Re Y_{l,m}, & m>0
\end{array}
\right.
\label{eq:checkY:def}
,
\end{equation}}
which retain simple rotational properties. Both $Y_{l,m}$ and $\check Y_{l,m}$ are orthonormal systems of functions.

We will apply the projection operator to $\check Y_{l,m}$ in order to get the desired basis functions that satisfy the four goals of
Section~\ref{sec:intro}, but will describe our results in terms of $Y_{l,m}$, because
much standard software is available. Standard computations based on the properties described in the previous
paragraph result in Lemma~\ref{lemma:projectionapplied}.
\begin{lemma}
	\label{lemma:projectionapplied}
	Suppose that the $p^{th}$ irrep of a group $G$ is potentially real with the
	real-valued orthogonal irrep matrices
	$\Gamma_r^p(g)\in\mathbb{R}^{d_p\times d_p}$ for all $g\in G$.
	Then, the projection operation on real spherical harmonics $\check Y_{l,m}$
	for $m\in\{-l,\dots,l\}$ and $l\in\mathbb{N}$ can be determined by
	\begin{eqnarray}
	\calP_{j,k}^p
	\check Y_{l,m}
	&=&
	\sum_{\m=-l}^{+l}
	\hat\calD_{j,k;l,m;\m}^p
	Y_{l,\m}(\theta,\phi)
	\label{eq:calP:checkY:Y}
	\end{eqnarray}
	where
	\begin{eqnarray}
	\hat\calD_{j,k;l,m;\m}^p
	&=&
	\frac{d_p}{\Ngroup}
	\sum_{g\in G}
	(\Gamma_r^p(g))_{j,k}
	\hat D_{l,m,\m}(R_g)
	\label{eq:hatcalD:def} \mbox{,~and}
	\\
	\hat D_{l,m,\m}
	&=&
	\left\{
	\begin{array}{ll}
	-
	\frac{i}{\sqrt{2}}
	\left(
	D_{l,m,\m}
	-
	(-1)^m
	D_{l,-m,\m}
	\right)
	,
	& m<0 \\
	D_{l,0,\m}
	,
	& m=0 \\
	\frac{1}{\sqrt{2}}
	\left(
	D_{l,m,\m}
	+
	(-1)^m
	D_{l,-m,\m}
	\right)
	,
	& m>0
	\end{array}
	\right.
	\label{eq:hatD:explicit}
	.
	\end{eqnarray}
\end{lemma}
\par
An alternative view of Lemma~\ref{lemma:projectionapplied} is described in
this paragraph.
Define the vectors $Y_l=(Y_{l,-l},\dots,Y_{l,+l})^T\in\mathbb{C}^{2l+1}$ and
$\check Y_l=(\check Y_{l,-l},\dots,\check Y_{l,+l})^T\in\mathbb{R}^{2l+1}$.
There exists a unitary matrix
$U_l\in\mathbb{C}^{(2l+1)\times(2l+1)}$
such that
$
\check Y_l
=
U_l^H Y_l
$
where $U_l$ has at most two non-zero entries in any row or any column.
The Wigner $D$ coefficients can be grouped into a matrix
$D_l(R)\in\mathbb{C}^{(2l+1)\times (2l+1)}$
such that
$
P(R)
Y_l
=
D_l(R)
Y_l
$
where $D_l(R)$ is typically a full matrix.
In terms of these two matrices, 
$
P(R)
\check Y_l
=
\hat D_l(R)
Y_l
$
where $\hat D_l(R)\in\mathbb{C}^{(2l+1)\times (2l+1)}$ is defined by
$\hat D_l(R)=U_l^H D_l(R)$,
The matrix equation $\hat D_l(R)=U_l^H D_l(R)$ 
is equivalent to Eq.~\ref{eq:hatD:explicit},  
but Eq.~\ref{eq:hatD:explicit} 
is less expensive to compute because of the sparseness of $U_l$.
\par
According to~\cite[p.~94]{Cornwell1984}, a vector of $d_p$ real basis functions, denoted by
$\calC_{k,l,m}^p\in\mathbb{R}^{d_p}$ and expressed in terms of
$Y_{l,m}$, can be computed from
Lemma~\ref{lemma:projectionapplied} (Eq.~\ref{eq:calP:checkY:Y}) as 
{\small
	\begin{align}\label{eq:VectorBasisFunction2}
	\calC_{k,l,m}^p(\theta,\phi)
	& =
	\frac{1}{\hat{c}^p_{k,l,m}}
	\Biggl[\begin{smallmatrix}
	\calP_{1,k}^p \check Y_{l,m}(\theta,\phi)\\
	\vdots\\
	\calP_{d_p,k}^p \check Y_{l,m}(\theta,\phi)
	\end{smallmatrix}\Biggr]=
	\hat{\boldsymbol{\calD}}_{l,m}^p Y_{l}(\theta,\phi),
	\end{align}
}
where
$(\hat{\boldsymbol{\calD}}_{l,m}^p)_{j,\m}=\hat \calD_{j,k,l,m,\m}^p/\hat{c}^p_{k,l,m}$
for $j\in\{1,\dots,d_p\}$, $\m\in\{-l,\dots,l\}$, and $\hat{c}^p_{k,l,m}=\sqrt{\sum_{\m=-l}^{l}|{\hat{\calD}}_{k,k,l,m,\m}^p|^2}$ all for some $k\in\{1,\dots,d_p\}$ such that $\hat{c}^p_{k,l,m}>0$.

Note that this procedure computes $2l+1$ coefficient matrices $\hat{\boldsymbol{\calD}}_{l,m}^p$ by varying
$m$ through the set $\{-l,\dots,+l\}$, so that a total of $(2l+1)d_p$ basis
functions are computed. This is more than necessary for a basis, because the subspace of square-integrable functions on the surface of the sphere, where the subspace is defined by degree $l\in\mathbb{N}$, is spanned by
$(2l+1)$ basis functions.
Through Gram-Schmidt orthogonalization, the set of coefficient matrices,
$\{\boldsymbol{\hat\calD}_{l,m}^p\}$ for $m\in\{-l,\dots,l\}$, shrinks to a
smaller set of coefficient matrices, $\{\boldsymbol{\hat\calH}_{l,n}^p\}$ for
$n\in\{1,\dots,\Npl<2l+1\}$.
The value of $\Npl\in\mathbb{N}$ is determined by this process.
Finally, the expression for the vector of $d_p$ orthonormal real basis
functions, $\basispln(\theta,\phi)$, is
\begin{equation}\label{eq:MatrixBasisFunction}
\basispln(\theta,\phi)
=
\boldsymbol{\hat\calH}_{l,n}^p Y_l(\theta,\phi),
\text{~for $n\in\{1,\dots,\Npl\}$.}
\end{equation} 
\par
\textcolor{black}{
Note that given the \texttt{WignerD} solutions calculated by \emph{Mathematica}~\cite{MathematicaURL}, computing the set of coefficient matrices,
$\boldsymbol{\hat\calD}_{l,m}^p$ for $m\in\{-l,\dots,l\}$, requires 
$2d_p(2l+1)[(3N_g+1)l+N_g]$ arithmetic operations. The Gram-Schmidt procedure, which shrinks the matrix of size $d_p(2l+1)\times(2l+1)$ to a matrix of size $d_pN_{p;l}\times(2l+1)$, requires $2d_p(2l+1)^3$ arithmetic operations~\cite[p.~255]{GolubVanLoan2013}. Eq.~\ref{eq:MatrixBasisFunction} for all $n$'s 
takes another $2d_pN_{p;l}l$ arithmetic operations, where $\sum_{p=1}^{\Nirrep}d_pN_{p;l}=2l+1$. 
Hence, given $l$, for all $p$'s ($p\in\{1,\dots,\Nirrep\}$) and all $n$'s ($n\in\{1,\dots,N_{p;l}\}$), it requires $2(2l+1)[(3N_g+1)l+N_g]\sum_pd_p+2(2l+1)^3\sum_pd_p+2l(2l+1)$ arithmetic operations for computing the set of basis functions $\basispln(\theta,\phi)$, and therefore the computational complexity is $\mathcal{O}((\sum_{p=1}^{\Nirrep}d_p)l^2(l+N_g))$.
 }
\par
Note that the basis is not unique.
In the approach of this paper, the nonuniqueness enters in several places,
e.g., in the choice of $A^p$ (Eq.~\ref{eq:Z:def}),
in the definition of the eigenvectors and the order of the loading of the
eigenvectors into the matrix $U$ (both Lemma~\ref{lemma:eigenVSconeigen}),
and in the creation of an orthonormal family of basis functions in the
subspace of dimension $2l+1$ which is spanned by the $2l+1$ spherical
harmonics of degree $l$.

\section{Application to the polyhedral groups}
\label{sec:application}
In this section, the theory of this paper is applied to the three polyhedral groups, which are the tetrahedral $T$, octahedral
$O$, and icosahedral $I$ groups. \textcolor{black}{Recall the fact that the tetrahedral group $T$ is the rotational symmetry group of the regular tetrahedron; the octahedral group $O$ is the rotational symmetry group of the cube and the regular octahedron; and the icosahedral group $I$ is the rotational symmetry group of the regular dodecahedron and the regular icosahedron.} Properties of each group and the parameter values which select a specific
basis are described in Section~\ref{sec:platonicsolids} and the numerical
results are presented in Section~\ref{sec:numericalresults}.
\par
\subsection{Irreps and rotation matrices of polyhedral groups} \label{sec:platonicsolids}
Unitary complex-valued irrep matrices for the tetrahedral and octahedral
groups are available at the Bilbao Crystallographic
Server~\cite{AroyoKirovCapillasPerezMatoWondratschekBILBAOActaCryst2006,TetrahedralURL,OctahedralURL}.
Unitary complex-valued irrep matrices for the icosahedral group are
provided by~\cite{LiuPingChenJMathPhys1990}.
The Frobenious-Schur indicator (Section~\ref{sec:computingrealirrepmatrices})
implies that all irreps of the
octahedral and the icosahedral groups are potentially real.
Similarly, the tetrahedral group has irreps $A$ and $T$ that are
potentially real and irreps ${}^1E$ and ${}^2E$ that are essentially
complex.
In the reminder of the paper, we refer to the tetrahedral irreps (the irreps of the tetrahedral group) $A$,
${}^1E$, ${}^2E$ and $T$ as the 1st, 2nd, 3rd and 4th irreps, respectively,
and refer to the octahedral irreps $A_1$, $A_2$, $E$, $T_1$ and $T_2$ as
the 1st, 2nd, 3rd, 4th and 5th irreps, respectively\footnote{\textcolor{black}{\{$A$,
${}^1E$, ${}^2E$, $T$\} and \{$A_1$, $A_2$, $E$, $T_1$, $T_2$\} are the Mulliken symbols used to identify irreps of group $T$ in \cite{TetrahedralURL}, and irreps of group $O$ in \cite{OctahedralURL}, respectively. $A$, $E$ and $T$ denote 1-dimensional, 2-dimensional, and 3-dimensional irrep, respectively. Note that there are two 2-dimensional irreps for the tetrahedral group, which are denoted by ${}^1E$ and ${}^2E$, respectively, in \cite{TetrahedralURL}. The irrep $A$ has symmetry with respect to rotation of the principle axis. $(\cdot)_1$ ($(\cdot)_2$) denotes the irrep which has symmetry (anti-symmetry) with respect to a vertical mirror plane perpendicular to the principal axis.
}}
The basic properties of the groups are tabulated in
Table~\ref{tab:polyhedralgp}.
\begin{table}[h]
	\begin{center}
		\begin{tabular}{|c|c|c|c|c|}
			\hline
			Symmetry Groups & $\Ngroup$ & $\Nirrep$ & $d_p$ & potentially real irreps \\
			\hline
			Tetrahedral & 12 & 4 & \{1, 1, 1, 3\} & 1,4\\
			\hline
			Octahedral & 24 & 5 & \{1, 1, 2, 3, 3\} & 1,2,3,4,5 \\
			\hline
			Icosahedral & 60 & 5 & \{1, 3, 3, 4, 5\} & 1,2,3,4,5\\
			\hline
		\end{tabular}
	\end{center}
	\caption{
		\label{tab:polyhedralgp}
		Basic properties of the polyhedral groups: the group orders
		($\Ngroup$), the number of irreps ($\Nirrep$), the dimension of the $p$th
		irrep ($d_p$ for $p\in\{1,\dots,\Nirrep\}$), and the potentially real irreps
		of each group.
	}
\end{table}
\par
For each symmetry operation, a rotation matrix
($R_g\in\mathbb{R}^{3\times 3}$ for $g\in G$ which satisfies
$R_g^T=R_g^{-1}$, $\det R_g=+1$) is needed.
The set of rotation matrices defines the relationship between the
symmetries and the coordinate system.
Any orthonormal real-valued irrep with $d_p=3$ can serve as such a set of
rotation matrices.
For the tetrahedral and octahedral groups, rotation matrices are available
at the Bilbao Crystallographic
Server~\cite{AroyoKirovCapillasPerezMatoWondratschekBILBAOActaCryst2006,TetrahedralURL,OctahedralURL}
although the matrices must be re-ordered in order to match the
multiplication tables of the irrep matrices and, after reordering, they are
the 4th irrep of the tetrahedral group and the 4th irrep of the octahedral
group.
For the icosahedral group, we desire to use the coordinate system in which
the $z$-axis passes through two opposite vertices of the icosahedron and
the $xz$ plane includes one edge of the
icosahedron~\cite{LaporteZNaturforschg1948,AltmannPCambPhilSoc1957,ZhengDoerschukTln1997}.
Rotation matrices in this coordinate system are
available~\cite{ZhengDoerschukComputersPhysics1995} although the matrices
must be reordered to match the multiplication table of the irrep
matrices~\cite{LiuPingChenJMathPhys1990}.
The reordering and the similarity matrix to match the rotation matrices to
either of the two $d_p=3$ sets of irrep matrices are given in
Appendix~\ref{sec:icosahedralthreedimensionalirreps}.
The calculations described in this paper use the rotation matrices
reordered to match the multiplication table of the 2nd irrep.
\par
For the particular irreps described above, it is necessary to give values
for the $A^p$ matrices of Lemma~\ref{lemma:Z:def}.
The identity matrix $I_{d_p}$ satisfies the nonsingular and transpose
symmetric hypotheses of Lemma~\ref{lemma:Z:def}.
However, for the $p=4$ irrep of the icosahedral group for which $d_4=4$,
$I_4$ leads to $Z^4=0$ by direct computation.
It was not difficult to find a choice for $A^p$ such that all
potentially-real irreps of the tetrahedral, octahedral, and icosahedral
groups have nonzero $Z^p$.
For instance, the choice of an ``exchange permutation''
matrix~\cite[Section~1.2.11, p.~20]{GolubVanLoan2013} for $A^p$, which is
the anti-diagonal matrix with all ones on the anti-diagonal, leads to
$Z^p=A^p$ by direct computation.
This choice for $A^p$ was used in all computations in this paper.

\subsection{Numerical results}
\label{sec:numericalresults}
For the tetrahedral group, the coefficient matrices $\boldsymbol{\hat\calH}_{l,n}^p$ for degree $l\in\{1,\dots,45\}$,
$p\in\{1,4\}$ and $n\in\{1,\dots, \Npl\}$, were computed.
The total number of rows in the coefficient matrices is
$N_{p=1;l}+N_{p=4;l}<2l+1$ for each $l$, which is in agreement with the
fact that only two of four irreps are potentially real and therefore only
two of four irreps are included in our calculation.
The resulting basis functions have been numerically verified to be real-valued and orthonormal.
\par
For the octahedral and icosahedral cases, there are numerical checks that can be performed on the basis functions because all irreps are potentially real.
Eq.~\ref{eq:IfromY} is achieved by construction.
Eq.~\ref{eq:I:isreal} is achieved by construction for $\hat\calH_{l,m}^p$.
To verify Eq.~\ref{eq:I:orthonormal}, form matrix
{\footnotesize$\boldsymbol{\hat\calH}_l=[(\boldsymbol{\hat\calH}_{l,1}^{p=1})^T,\dots,(\boldsymbol{\hat\calH}_{l,\Npl})^T, \dots, (\boldsymbol{\hat\calH}_{l,1}^{p=\Nirrep})^T,\dots, (\boldsymbol{\hat\calH}_{l,\Npl}^{p=\Nirrep})^T]^T$}. The matrix dimension is verified to be $(2l+1)\times(2l+1)$, which verifies that the correct number of basis functions have been found
($\sum_{p=1}^{\Nirrep} d_p N_l^p=2l+1$). Moreover, the matrix $\boldsymbol{\hat\calH}_l$
is verified to be unitary, which verifies that the basis functions are orthonormal.
Eq.~\ref{eq:PFprop} is verified by testing an array of $(\theta,\phi)$
values.
The verifications were carried out for $l\in\{0,\dots,45\}$.
\begin{figure}[H]
	\begin{center}
		\begin{tabular}{cc}
			\includegraphics[width=1.8cm]{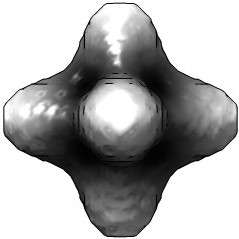}
			&
			\includegraphics[width=1.8cm]{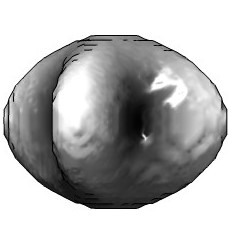}
			\\
		\end{tabular}
		\\
		(a) Tetrahedral basis functions $T_{p,l,n,j}$
		\\
		\begin{tabular}{ccccc}
			\includegraphics[width=1.8cm]{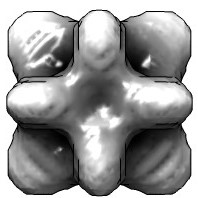}
			&
			\includegraphics[width=1.8cm]{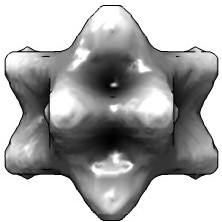}
			&
			\includegraphics[width=1.8cm]{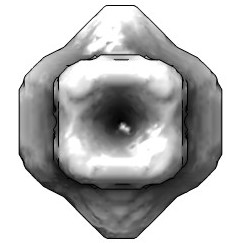}
			&
			\includegraphics[width=1.8cm]{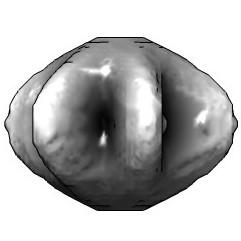}
			&
			\includegraphics[width=1.8cm]{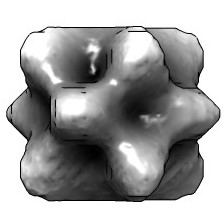}
		\end{tabular}
		\\
		(b) Octahedral basis functions $O_{p,l,n,j}$
		\\
		\begin{tabular}{ccccc}
\includegraphics[width=1.9cm]{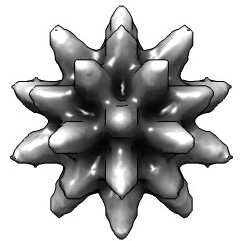}
&
\includegraphics[width=1.8cm]{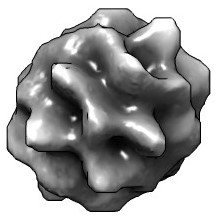}
&
\includegraphics[width=1.8cm]{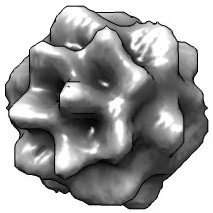}
&
\includegraphics[width=1.8cm]{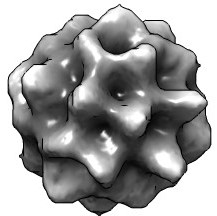}
&
\includegraphics[width=1.8cm]{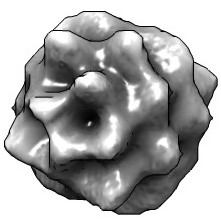}
		\end{tabular}
		\\
		(c) Icosahedral basis functions $I_{p,l,n,j}$
	\end{center}
	\caption{
		\label{fig:irredrep}
		Examples of the real basis functions of the three polyhedral
		groups.
		The surfaces of 3-D objects defined by Eq.~\ref{eq:visualizationofharmonic}
		are visualized by UCSF
		Chimera~\cite{PettersenHuangCouchGreenblattMengFerrin2004} where the darkness indicates the distance from the center of the object. The darker
		the color is, the closer the point is to the center.
	}
\end{figure}
\par
Example basis functions are shown in
Figure~\ref{fig:irredrep} by visualization of the function
\begin{equation}
\xi_{p,l,n,j}(\vx)
=
\left\{
\begin{array}{ll}
1, & x\le \kappa_1 + \kappa_2 F_{p,l,n,j}(\vx/x) \\
0, & \mbox{otherwise}
\end{array}
\right.
\label{eq:visualizationofharmonic}
\end{equation}
where $\kappa_1$ and $\kappa_2$ are chosen so that
$0.5\le
\kappa_1 + \kappa_2 I_{p,l,n,j}(\vx/x)
\le 1$. Software and numerical solution for the real irrep matrices and spherical harmonics coefficient for constructing real basis functions for theses three polyhedral groups are available here (\url{https://github.com/nxu25/PolyhedralBasisFunction}) as well as described in the Supplemental Materials.

\section{Conclusion}
Motivated by cryo electron microscopy problems in structural biology,
this paper presents a method for computing real-valued basis functions
which transform as the various rows and irreducible representations of a polyhedral group.
The method has two steps: (1)~compute real-valued orthonormal irreducible
representation matrices (Section~\ref{sec:computingrealirrepmatrices}) and
(2)~use the matrices to define projection operators which are applied to a
real-valued basis for the desired function space
(Section~\ref{sec:computebasisfunctions}).
The method is applied to the icosahedral, octahedral, and tetrahedral
groups where the second step is performed in spherical coordinates using
the spherical harmonics basis.
The most burdensome part of the calculation for the first step is the
solution of a real symmetric eigenvector problem of dimension equal to
twice the dimension of the irreducible representation matrices.
For these three groups, the largest matrix is of dimension 5 so the
calculations are straightforward.
Of the remaining polyhedral groups, basis functions for the
cyclic groups are more naturally described in cylindrical coordinates
using the complex exponential basis and possibly the same is true
for the dihedral groups and so the calculations for the second step would
be quite different from those described in this paper.
However, the calculations in the first step, which apply to any potentially
real irreducible representation, would remain relevant.
\par
The resulting basis functions are described by linear combinations of
spherical harmonics and a \emph{Mathematica}~\cite{MathematicaURL} program
to compute the coefficients of the linear combination and a
\emph{Matlab}~\cite{MatlabURL} program to evaluate the resulting basis
functions are on our \href{https://github.com/nxu25/PolyhedralBasisFunction}{github page}.

\section*{Acknowledgments}
We are very grateful for helpful discussions with Prof.\ Dan M. Barbasch
(Department of Mathematics, Cornell University) about representation theory and the method for generating real irreducible representation matrices in Section~\ref{sec:computingrealirrepmatrices}. 
In particular, the idea of formulating Eq.~\ref{eq:Z:def} (Lemma \ref{lemma:Z:def}) was contributed by Prof. Barbasch. We also thank the financial support from the National Science Foundation under grant number 1217867, and N.X. thanks the financial support from Georgia Institute of Technology under the postdoctoral fellowship.

\bibliographystyle{siamplain}
\bibliography{references}

\clearpage

\appendix
\section{Proofs of Lemmas}
\label{sec:appendixA}
\begin{proof}[Proof of Lemma~\ref{lemma:SStrans}]
	Eq.~\ref{eq:complex2real} implies Eq.~\ref{eq:GammaSimilarToGammaConjugate})
	$\Gamma_r^p$ is real by definition so that
	\[\Gamma_r^p(g)=(\Gamma_r^p(g))^\ast\].
	Since $\Gamma_r^p=(S^p)^H\Gamma_c^p(g)S^p$, it follows that
	\[\Gamma_r^p=
	(S^p)^H
	\Gamma_c^p(g)
	S^p
	=
	\left(
	(S^p)^H
	\Gamma_c^p(g)
	S^p
	\right)^\ast
	=
	(S^p)^T
	(\Gamma_c^p(g))^\ast
	(S^p)^\ast.\]
	Multiply on the left by
	$({(S^p)^H})^T=({(S^p)^{-1}})^T=({(S^p)^T})^{-1}$
	and on the right by
	$({(S^p)^H})^\ast=(S^p)^T$
	to get
	\[({(S^p)^T})^{-1}(S^p)^H\Gamma_c^p(g)S^p(S^p)^T=(\Gamma_c^p(g))^\ast,\]
	which, since $(S^p)^H=(S^p)^{-1}$, implies that
	\[[S^p(S^p)^T]^{-1}\Gamma_c^p(g)[S^p(S^p)^T]=(\Gamma_c^p(g))^\ast.\]
	Therefore $\Gamma_c^p(g)$ is similar to $(\Gamma_c^p(g))^\ast$.
	
	Eq.~\ref{eq:GammaSimilarToGammaConjugate} implies Eq.~\ref{eq:complex2real})
	Multiplying by $(S^p)^T$ on the left and $(S^p)^\ast$ on the right of
	Eq.~\ref{eq:GammaSimilarToGammaConjugate} gives
	\[(S^p)^T[S^p(S^p)^T]^{-1}\Gamma_c^p(g)[S^p(S^p)^T](S^p)^\ast=(S^p)^T(\Gamma_c^p(g))^\ast(S^p)^\ast\]
	which can be reorganized using the assumption that $S^p$ is unitary to get
	\[[(S^p)^T(S^p)^\ast](S^p)^{-1}\Gamma_c^p(g)S^p[(S^p)^T(S^p)^\ast]=(S^p)^T(\Gamma_c^p(g))^\ast(S^p)^\ast.\]
	Then, also since $S^p$ is unitary, it follows that
	\[(S^p)^{-1}\Gamma_c^p(g)S^p=\left[(S^p)^{-1}\Gamma_c^p(g)S^p\right]^\ast.\]
	Since the left and the right hand sides of the above equation are complex
	conjugates of each other, it follows that each is real, i.e.,
	$\Gamma_r^p(g)=(S^p)^{-1}\Gamma_c^p(g)S^p$ is a real-valued matrix.
\end{proof}
\par
\begin{proof}[Proof of Lemma~\ref{lemma:Z:def}]
	Property~1:
	Because the irrep is unitary, $Z^p$ can be written in the form
	\begin{equation}
	Z^p
	=
	\frac{1}{\Ngroup}
	\sum_{g\in G}
	\Gamma_c^p(g)
	A^p
	({\Gamma_c^p(g)})^T.
	\label{eq:Zp:defwithtranspose}
	\end{equation}
	Then, Property~1 follows from a direct computation.
	\par
	Properties~2 and~3:
	For any arbitrary $g^\prime\in G$, we have
	\begin{eqnarray}
	\Gamma_c^p(g^\prime)
	Z^p
	({\Gamma_c^p(g^\prime)})^T
	&=&
	\Gamma_c^p(g^\prime)
	\frac{1}{\Ngroup}
	\sum_{g\in G}
	\Gamma_c^p(g)
	A^p
	({\Gamma_c^p(g)})^T
	({\Gamma_c^p(g^\prime)})^T\nonumber
	\\
	&=&
	\frac{1}{\Ngroup}
	\sum_{g\in G}
	[
	\Gamma_c^p(g^\prime)
	\Gamma_c^p(g)
	]
	A^p
	[
	\Gamma_c^p(g^\prime)
	{\Gamma_c^p(g)})
	]^T\nonumber
	\\
	&=&
	\frac{1}{\Ngroup}
	\sum_{g\in G}
	\Gamma_c^p(g^\prime g)
	A^p
	({\Gamma_c^p(g^\prime g)})^T
	s=
	\frac{1}{\Ngroup}
	\sum_{g\in G}
	\Gamma_c^p(g)
	A^p
	({\Gamma_c^p(g)})^T
	\label{eq:UsedRearrangementTheorem}
	=
	Z^p
	\label{eq:GammaGammaTransposeEqualsZ}
	\end{eqnarray}
	where the forth equivalence follows from the Rearrangement
	Theorem~\cite[Theorem~II, p.~24]{Cornwell1984}.
	Because the irrep is unitary, rearranging Eq.~\ref{eq:GammaGammaTransposeEqualsZ} gives $\Gamma_c^p(g^\prime)
	Z^p
	=
	Z^p
	({\Gamma_c^p(g^\prime)})^\ast$.
	Because $g^\prime$ is arbitrary,
	\begin{equation}
	\Gamma_c^p(g)
	Z^p
	=
	Z^p
	({\Gamma_c^p(g)})^\ast, \text{~for all $g\in G$.}
	\label{eq:Zissimilaritymatrix}
	\end{equation}
	Property~2 follows from Ref.~\cite[Theorem~II, p.~128]{Cornwell1984} because the irrep $\Gamma_c^p$ is potentially real.
	\par
	Note that both $\Gamma_c^p$ and $(\Gamma_c^p)^\ast$ are unitary irreps of
	dimension $d_p$ of the group $G$.
	Schur's Lemma~\cite[Theorem~I, p.~80]{Cornwell1984} applied to
	Eq.~\ref{eq:Zissimilaritymatrix} implies that either $Z^p=0$ or $\det
	Z^p\neq 0$.
	Because of the assumption $Z^p\neq 0$, $Z^p$ is nonsingular.
	Therefore, $Z^p$ is a similarity transform from ${\Gamma_c^p(g)}$ to
	$({\Gamma_c^p(g)})^\ast$ for all $g\in G$ which proves Property~3.
\end{proof}
\par
\begin{proof}[Proof of Lemma~\ref{lemma:eigenVSconeigen}]
	For simplicity, let $Q_1=\Re Q$ and $Q_2=\Im Q$.
	\par
	Property~1: The matrices $Q_1$, $Q_2$, and $B$ are all real and symmetric.
	Since $B\in\mathbb{R}^{2n\times 2n}$ and $B^T=B$, $B$ has $2n$ real eigenvalues
	(possibly repeated) and $2n$ real orthonormal eigenvectors~\cite[Theorem~2.5.6, p.~104]{HornJohnson1985}.
	Define $M$ by
	\[
	M
	=
	\Bigl[
	\begin{array}{cc}
	I & -i I \\
	0 & I
	\end{array}
	\Bigr]
	B
	\Bigl[
	\begin{array}{cc}
	I & 0 \\
	iI & I
	\end{array}
	\Bigr]
	=
	\Bigl[
	\begin{array}{cc}
	0 & Q_2 + iQ_1 \\
	Q_2-iQ_1 & -Q_1
	\end{array}
	\Bigr].\]
	Then,\[
	\det(B)=\det(M)=\det((Q_2 + iQ_1)(Q_2 - iQ_1)-0(-Q_1)=\det(QQ^\ast)
	=|\det(Q)|^2>0\]
	because $Q$ is non-singular. Hence, $B$ is non-singular.
	\par
	Property~2:
	\[
	B
	\left[
	\begin{smallmatrix}
	x \\
	-y
	\end{smallmatrix}
	\right]
	=
	\lambda
	\left[
	\begin{smallmatrix}
	x \\
	-y
	\end{smallmatrix}
	\right]\]
	\[\Longleftrightarrow\begin{cases}
	Q_1 x - Q_2 y &= \lambda x \\
	Q_2 x + Q_1 y &= -\lambda y
	\end{cases}\Longleftrightarrow\begin{cases}
	Q_2 y - Q_1 x &= -\lambda x \\
	Q_2 x + Q_1 y &= -\lambda y
	\end{cases}\]
	\[\Longleftrightarrow B
	\left[
	\begin{smallmatrix}
	x \\
	y
	\end{smallmatrix}
	\right]
	=
	-\lambda
	\left[
	\begin{smallmatrix}
	x \\
	y
	\end{smallmatrix}
	\right].\]
	\par
	Property~3:
	Define the matrices
	$X=\left[x_1,\dots, x_n\right]\in\mathbb{R}^{n\times n}, Y
	=
	\left[
	y_1,\dots, y_n
	\right]
	\in\mathbb{R}^{n\times n}, 
	\Lambda
	=
	\diag(\lambda_1,\dots,\lambda_n)
	\in\mathbb{R}^{n\times n},
	U
	=
	X-iY
	\in\mathbb{C}^{n\times n}$.
	Then, the equation
	\[
	B
	\left[
	\begin{smallmatrix}
	x_k \\
	-y_k
	\end{smallmatrix}
	\right]
	=\lambda_k
	\left[
	\begin{smallmatrix}
	x_k \\
	-y_k
	\end{smallmatrix}
	\right]~\mbox{for $k\{1,\dots,n\}$}
	\]
	is equivalent to
	$B
	\left[
	\begin{array}{cc}
	X \\
	-Y
	\end{array}
	\right]
	=
	\left[
	\begin{array}{cc}
	Q_1 & Q_2 \\
	Q_2 & -Q_1
	\end{array}
	\right]
	\left[
	\begin{array}{cc}
	X \\
	-Y
	\end{array}
	\right]
	=
	\left[
	\begin{array}{cc}
	X \\
	-Y
	\end{array}
	\right]\Lambda$, 
	which is equivalent to
	$\begin{cases}
	Q_1 X - Q_2 Y &= X\Lambda \\
	Q_2 X + Q_1 Y &= -Y\Lambda
	\end{cases}.$
	\par
	Multiplying the second equation by $i$ and adding to the first equation
	gives
	\begin{align}
	U\Lambda&=(X-iY)\Lambda=(Q_1 X - Q_2 Y)+i(Q_2 X + Q_1 Y)\nonumber\\
	&=(Q_1 + iQ_2)X+(iQ_1 - Q_2)Y=(Q_1+iQ_2)X+(Q_1+iQ_2)iY\nonumber\\
	&=(Q_1+iQ_2)(X+iY)=Q U^\ast.\nonumber
	\end{align}
	Therefore $x_k-iy_k$ and $+\lambda_k$ are the coneigenvectors and
	coneigenvalues of $Q$, respectively.
	\par
	Property~4:
	Because $Q Q^\ast=I_n$ by assumption, the eigenvalues of $Q Q^\ast$
	are the eigenvalues of $I_n$ which all have value 1.
	By Ref.~\cite[Proposition~4.6.6, p.~246]{HornJohnson1985}, $\xi$ is an
	eigenvalue of $Q Q^\ast$ if and only if $+\sqrt{\xi}$ is a
	coneigenvalue of $Q$. Therefore, all the coneigenvalues of $Q$ have value 1.
	\par
	Property~5:
	Let the columns of $V\in\mathbb{R}^{2n\times 2n}$ be the $2n$ real orthonormal
	eigenvectors of $B$, i.e.,
	\begin{equation}
	V
	=
	\left[
	\left[
	\begin{array}{c}
	x_1 \\
	-y_1
	\end{array}
	\right]
	,
	\dots
	,
	\left[
	\begin{array}{c}
	x_n \\
	-y_n
	\end{array}
	\right]
	,
	\left[
	\begin{array}{c}
	y_1 \\
	x_1
	\end{array}
	\right]
	,
	\dots
	,
	\left[
	\begin{array}{c}
	y_n \\
	x_n
	\end{array}
	\right]
	\right]
	.
	\end{equation}
	Then, $V^T V = V V^T=I_{2n}$ and $V^H V = V V^H=I_{2n}$.
	\par
	Define $L\in\mathbb{C}^{n\times 2n}$ by $L=[I_n , iI_n]$ and $\tilde
	U\in\mathbb{C}^{2\times 2n}$ by $\tilde U=LV$.
	Then
	$\tilde U \tilde U^H = (LV)(LV)^H = L V V^H L^H = L I_{2n} L^H = L L^H = I_n+I_n=2I_n$.
	But also,
	$\tilde U
	=
	LV
	=
	[x_1-iy_1,\dots,x_n-iy_n, y_1+ix_1,\dots,y_n+ix_n]
	=
	[x_1-iy_1,\dots,x_n-iy_n, i(x_1-iy_1),\dots,i(x_n-iy_n)]
	=
	[U, iU]$
	and
	$\tilde U \tilde U^H
	=
	[U, iU] \left[\begin{array}{c} U^H \\ -iU^H \end{array}\right]
	=
	U U^H + U U^H
	=
	2U U^H$.
	Therefore, $U U^H=I_n$.
	\par
	Property~6:
	Property~6 follows immediately from Properties~3--5 since
	Property~3 states that $QU^\ast=U\Lambda$, Property~4 states that
	$\Lambda=I_n$, and Property~5 states that $({U^\ast})^{-1}=U^T$.
\end{proof}
\section{Relationships between icosahedral $d_p=3$ irreps}
\label{sec:icosahedralthreedimensionalirreps}
Let $R_g$ be the rotation matrices of
Ref.~\cite{ZhengDoerschukComputersPhysics1995} which are also a real
orthonormal irrep of dimension 3.
Let $\Gamma^p(g)$ be the complex unitary irreps of
Ref.~\cite{LiuPingChenJMathPhys1990} where $p=2$ and $p=3$ are of dimension
3.
With different permutations, $R_g$ can be made similar to both
$\Gamma^{p=2}(g)$ and $\Gamma^{p=3}(g)$.
In particular,
$\Gamma^p(g)=(S^p)^H R_{\gamma^p(g)} S^p$ for $p\in\{2,3\}$ where the
permutation $\gamma^p(g)$ and the complex unitary matrices
$S^p\in\mathbb{C}^{3\times 3}$ are given in Table~\ref{table:permutations}
and Eq.~\ref{eq:Uforlinkingirreps}, respectively.
\begin{equation}
S^{p=2}
=
\left[
\begin{array}{ccc}
-1/\sqrt{2} & 0 & -1/\sqrt{2} \\
-i/\sqrt{2} & 0 & i/\sqrt{2} \\
0 & 1 & 0
\end{array}
\right]
\quad
S^{p=3}
=
\left[
\begin{array}{ccc}
-1/\sqrt{2} & 0 & -1/\sqrt{2} \\
i/\sqrt{2} & 0 & -i/\sqrt{2} \\
0 & 1 & 0
\end{array}
\right]
\label{eq:Uforlinkingirreps}
.
\end{equation}
\begin{table}[H]
	\begin{center}
		\begin{tabular}{l|llllllllllllllllllll}
			$g$ & 1 & 2 & 3 & 4 & 5 & 6 & 7 & 8 & 9 & 10 & 11 & 12 & 13 & 14 & 15  \\
			$\gamma^2(g)$ & 1 & 2 & 5 & 9 & 17 & 10 & 27 & 13 & 21 & 18 &24 & 15 & 26 & 3 & 4 \\
			$\gamma^3(g)$ & 1 & 4 & 3 & 36 & 52 & 38 & 42 & 49 & 60 & 54 & 56 & 48 & 45 & 2 & 5 \\
			\hline
			$g$ & 16 & 17 & 18 & 19 & 20 & 21 & 22 & 23 & 24 & 25 & 26 & 27 & 28 & 29 & 30 \\
			$\gamma^2(g)$ & 48 & 45 & 56 & 54 & 49 &60 & 36 & 52 & 42 & 38 & 14 & 16 & 47 & 40 & 46 \\
			$\gamma^3(g)$ & 24 & 18 & 15 & 26 & 21 & 13 & 10 & 27 & 17 & 9 & 46 & 55 & 22 & 8 & 25 \\
			\hline
			$g$ & 31 & 32 & 33 & 34 & 35 & 36 & 37 & 38 & 39 & 40 & 41 & 42 & 43 & 44 & 45\\
			$\gamma^2(g)$ &55 & 41 & 53 & 20 & 29 & 6 & 12 & 57 & 39 & 8 &22 & 44 & 58 & 28 & 25\\
			$\gamma^3(g)$ & 28 & 20 & 29 & 53 & 41 & 40 & 47 & 12 & 6 & 39 & 57 & 16 & 14 & 44 & 58 \\
			\hline
			$g$  & 46 & 47 & 48 & 49 & 50 & 51 & 52 & 53 & 54 & 55 & 56 & 57 & 58 & 59 & 60 \\
			$\gamma^2(g)$  & 11 & 31 & 59 & 33 & 30 &19 & 43 & 35 & 34 & 37 & 23 & 7 & 50 & 32 & 51\\
			$\gamma^3(g)$  & 50 & 31 & 11 & 32 & 43 & 51 & 19 & 33 & 35 & 7 & 59 & 37 & 23 & 34 & 30\\
		\end{tabular}
	\end{center}
	\caption{
		\label{table:permutations}
		Permutations relating the 3 dimensional icosahedral irreps of Refs.~\cite{ZhengDoerschukComputersPhysics1995,LiuPingChenJMathPhys1990}.
	}
\end{table}

\clearpage
\setcounter{page}{1} 

\noindent\textbf{Supplemental Materials:}
All files described in Supplemental Materials are available at \href{https://github.com/nxu25/PolyhedralBasisFunction}{https://github.com/nxu25/PolyhedralBasisFunction}.
\section*{I. Software for computing the real irrep matrices and real basis functions}
Software packages in \emph{Mathematica} were developed for computing the real irrep matrices as well as the spherical harmonics coefficients $c_{p,l,n,j,m}$ (Eq. 6.1) which define the real basis functions in terms of spherical harmonics for the three polyhedral groups. Specific software programs and functions for each group are listed in Table SM1. Finally, the real basis functions can be obtained by multiplying each row of $\boldsymbol{\hat\calH}_{l}^{p}$ by the spherical harmonics vector (i.e., \texttt{Table[SphericalHarmonicY[l,m,$\theta$,$\phi$],\{m,-l,l\}]} in \emph{Mathematica}). Please see the notebook file ``\textsf{\url{Main.nb}}'' for the tutorial of calling these packages to generate real basis function for each polyhedral group. 
\begin{table}[h!]
	\sffamily
	\centering
	\renewcommand\thetable{SM1} 
	\begin{tabularx}{\textwidth}{|X|X|X|X|}
		\hline
		Functions $\backslash$  Group & $T$ & $O$ & $I$ \\\hline 
	Software package
		& \small{\url{RealIrrepBasisT.m}}
		& \small{\url{RealIrrepBasisO.m}}
		& \small{\url{RealIrrepBasisI.m}}
		\\\hline 
		Irrep matrix
		&  $\mathtt{\Gamma t}$0[\emph{p\underline{~},g\underline{~}}]\small{$\in\mathbb{C}^{d_p\times d_p}$} \newline
		$\mathtt{\Gamma t}$[\emph{p\underline{~},g\underline{~}}] \small{$\in\mathbb{R}^{d_p\times d_p}$}
		& $\mathtt{\Gamma o}$0[\emph{p\underline{~},g\underline{~}}]\small{$\in\mathbb{C}^{d_p\times d_p}$}\newline 
		$\mathtt{\Gamma o}$[\emph{p\underline{~},g\underline{~}}]\small{$\in\mathbb{R}^{d_p\times d_p}$} 
		& $\mathtt{\Gamma}$0[\emph{p\underline{~},g\underline{~}}]\small{$\in\mathbb{C}^{d_p\times d_p}$} \newline
		$\mathtt{\Gamma r}$[\emph{p\underline{~},g\underline{~}}] \small{$\in\mathbb{R}^{d_p\times d_p}$} \\\hline
		\footnotesize{Non-orthogonalized coefficients $\boldsymbol{\hat\calD}_{l,m}^{p}$\newline (Eq.~6.10)}
		&\footnotesize{BasisRealFunctionCo-\newline
			effMatrixT[\emph{l\underline{~},m\underline{~},p\underline{~}}]}
		&\footnotesize{BasisRealFunctionCo- \newline
			effMatrixO[\emph{l\underline{~},m\underline{~},p\underline{~}}]}
		&\footnotesize{BasisRealFunctionCo- \newline
			effMatrixI[\emph{l\underline{~},m\underline{~},p\underline{~}}]}\\\hline
		\footnotesize{Matrix of coefficients  $\boldsymbol{\hat\calH}_l^p=
			\Bigg[\begin{smallmatrix}
			\boldsymbol{\hat\calH}_{l,n=1}^{p}\\
			\vdots\\
			\boldsymbol{\hat\calH}_{l,n=\Npl}^{p}
			\end{smallmatrix}\Bigg]$}\newline (Eq.~6.11)
		&\footnotesize{BasisRealFunctionOt-\newline
			hoCoeffMatrixT[\emph{l\underline{~},p\underline{~}}]}
		&\footnotesize{BasisRealFunctionOt-\newline
			hoCoeffMatrixO[\emph{l\underline{~},p\underline{~}}]}
		&\footnotesize{BasisRealFunctionOt-\newline
			hoCoeffMatrixI[\emph{l\underline{~},p\underline{~}}]}
		\\\hline
	\end{tabularx}
	\caption{\emph{Mathematica} software packages and functions for computing real irrep matrices and coefficient matrices of real basis functions for each group. Values \emph{\textsf{p}} $\in\{1,\dots,\Nirrep\}$, \emph{\textsf{g}} $\in\{1,\dots,N_g\}$, \emph{\textsf{m}} $\in\{-l,\dots,l\}$, \emph{\textsf{precisionN, l}} $\in\mathbb{N}$, $\theta\in[0,\pi]$, and $\phi\in[0,2\pi)$. Note that for the group $T$, only the p = 1 and p = 4 irreps are real valued and lead to real-valued functions.}\vspace{-.7cm}
\end{table}
\section*{II. Numerical Solutions}
Note that the solution of real irrep matrices and coefficients are not unique as described in Section 6. One solution for each group is included as is tabulated in Table SM2.
\begin{table}[h!]  
	\sffamily
	\centering
	\renewcommand\thetable{SM2} 
	\begin{tabularx}{\textwidth}{|X|l|l|l|X|}
		\hline
		Results & $T$ & $O$ & $I$ &Format\\\hline 
		{Real irrep matrices}
		& \small{RealIrreps\underline{~}T.txt}
		& \small{RealIrreps\underline{~}O.txt}
		& \small{RealIrreps\underline{~}I.txt}
		& \\\hline
		{\footnotesize$\boldsymbol{\hat\calH}_l^p$ \newline $0\leq l\leq100$}  
		& \small{BasisCoeff\underline{~}T.txt}
		& \small{BasisCoeff\underline{~}O.txt}
		& \small{BasisCoeff\underline{~}I.txt}
		& \footnotesize{a line of l value, a line of p value, and then a line of matrix $\boldsymbol{\hat\calH}_l^p$}
		\\\hline
		{\footnotesize$\boldsymbol{F}_{l}^p$ at random $(\theta,\phi)$'s\newline  $0\leq l\leq100$}
		& \footnotesize{RealBasisTest\underline{~}T.txt}
		& \footnotesize{RealBasisTest\underline{~}O.txt}
		& \footnotesize{RealBasisTest\underline{~}I.txt}
		& \footnotesize{a line of l value, a line of p value, and then a line of ``\{$\theta$,$\phi$\} $\boldsymbol{F}_{l}^p(\theta,\phi)$''}
		\\\hline
	\end{tabularx}
	\caption{Numerical results of real irrep matrices and coefficient matrices of real basis functions. In all these files, a matrix $\big(\begin{smallmatrix}
			a & b\\
			c & d
			\end{smallmatrix}\big)$ is in the form of \{\{a,b\},\{c,d\}\}.} \vspace{-.5cm}
\end{table}
\setlength\headheight{45pt}
\section*{III. Obtain real basis functions in \emph{Matlab}} \emph{Matlab} functions are also developed to read the coefficients file (i.e., \url{read_coefMat.m} for ``BasisCoeff\underline{~}*.txt”) and then to compute the real basis functions (i.e., \url{demonstrate_get_Fplnj.m} and \url{get_Fplnj.m}). 

\end{document}